\documentclass[leqno,10pt]{amsart}
\usepackage{amssymb,amsmath,amsfonts,amsthm,mathtools,latexsym,ifthen,bezier}
\usepackage{graphicx}
\usepackage{amsxtra}
\usepackage{xspace}
\usepackage[english]{babel} 

\usepackage{enumitem}
\usepackage[margin=1.3in]{geometry}
\usepackage{setspace}
\usepackage[english]{babel}
\usepackage{latexsym}
\usepackage{enumitem}
\usepackage{amsthm}
\usepackage{amssymb}
\DeclareMathAlphabet{\mathpzc}{OT1}{pzc}{m}{it}
\usepackage[latin1]{inputenc}
\usepackage{afterpage}

\usepackage{bbm}
\usepackage{tikz} 

\newtheorem{theorem}{Theorem}[section]
\newtheorem*{thm*}{Theorem}

\newtheorem{lemma}[theorem]{Lemma}
\newtheorem{proposition}[theorem]{Proposition}
\newtheorem{corollary}[theorem]{Corollary}

\newtheorem{claim}[theorem]{Claim}

\theoremstyle{definition}
\newtheorem{definition}[theorem]{Definition}

\theoremstyle{remark}
\newtheorem{remark}{Remark}

\newtheorem{question}{Question}

\def\hook{\upharpoonright}
\def\forces{\Vdash}

\newfont{\ssi}{cmssi12 at 12pt}

\newenvironment{ea*}{\begin{eqnarray*}}{\end{eqnarray*}}

\setbox0=\hbox{$\longrightarrow$}

\renewcommand{\phi}{\varphi}

\def\<#1>{\langle#1\rangle}

\newcommand{\MA}{\ensuremath{\mathsf{MA}}}

\newcommand{\MP}{\ensuremath{\mathsf{MP}}}

\newcommand{\ColNothing}{\mathrm{Col}}
\newcommand{\Col}[1]{\ColNothing(#1)}

\newcommand{\MPColNothing}[1]{\MP_{\Col{\dot{\kappa}}}}

\newcommand{\GCH}{\ensuremath{\mathsf{GCH}}\xspace}
\newcommand{\CH}{\ensuremath{\mathsf{CH}}\xspace}

\def\hook{\upharpoonright}
\def\forces{\Vdash}

\def\MA{\mathsf{MA}}

\def\mfb{\mathfrak b}

\def\mfp{\mathfrak{p}}
\def\GCH {\mathsf{GCH}}
\def\CH {\mathsf{CH}}

\title{The Structure of $\kappa$-Maximal Cofinitary Groups}
\author[Fischer]{Vera Fischer}
\address[V. ~Fischer]{Institut f\"{u}r Mathematik, Kurt G\"odel Research Center, Universit\"{a}t Wien, Kolingasse 14-16, 1090 Wien, AUSTRIA}
\email{vera.fischer@univie.ac.at}

\author[Switzer]{Corey Bacal Switzer}
\address[C.~B.~Switzer]{Institut f\"{u}r Mathematik, Universit\"{a}t Wien, Kurt G\"odel Research Center, Kolingasse 14-16, 1090 Wien, AUSTRIA}
\email{corey.bacal.switzer@univie.ac.at}

\thanks{\emph{Acknowledgements:} The authors would like to thank the
Austrian Science Fund (FWF) for the generous support through grant number Y1012-N35.}

\date{}


\begin{document}
\maketitle

\begin{abstract}
We study $\kappa$-maximal cofinitary groups for $\kappa$ regular uncountable, $\kappa = \kappa^{<\kappa}$. Revisiting earlier work of Kastermans and building upon a recently obtained higher analogue of Bell's theorem, we show that:
\begin{enumerate}
\item Any $\kappa$-maximal cofinitary group has ${<}\kappa$ many orbits under the natural group action of $S(\kappa)$ on $\kappa$.
\item If $\mfp(\kappa) = 2^\kappa$ then any partition of $\kappa$ into less than $\kappa$ many sets can be realized as the orbits of a $\kappa$-maximal cofinitary group.
\item  For any regular $\lambda > \kappa$ it is consistent that there is a $\kappa$-maximal cofinitary group which is universal for groups of size ${<}2^\kappa = \lambda$. If we only require the group to be universal for groups of size $\kappa$ then this follows from $\mfp(\kappa) = 2^\kappa$.
\end{enumerate}
\end{abstract}

\section{Introduction}
Given a set $X$ we denote by $S(X)$ the group of permutations of $X$. The group $S(\omega)$ and its subgroups have been of interest in set theory of the reals, both for their combinatorial, as well as descriptive set theoretic properties (see for example \cite{MADfamandNeighbors, KastermansZhang2012, BorelMCG}). Of particular interest are the {\em maximal cofinitary} subgroups. Here, a permutation $f \in S(\omega)$ is {\em cofinitary} if it has only finitely many fixed points. A group $G \leq S(\omega)$ is {\em cofinitary} if all of its non-identity elements are cofinitary. The group $G$ is moreover said to be {\em maximal} if no proper supergroup is cofinitary.

An important area of set theory, which has been of increased interest in the past decades, is the study of analogues of  combinatorial sets of reals in the higher Baire spaces, $\kappa^\kappa$ and $2^\kappa$. In this paper, we study such higher  counterparts to {\em maximal cofinitary groups}. Throughout the paper, we fix a regular, uncountable cardinal $\kappa$ and assume that $\kappa^{{<}\kappa} = \kappa$. In analogy with the countable case, a permutation $f \in S(\kappa)$ is said to be $\kappa$-cofinitary if it has ${<}\kappa$ many fixed points. A subgroup $G \leq S(\kappa)$ is said to be  $\kappa$-cofinitary if all of its non-identity elements are $\kappa$-cofinitary and  moreover {\em maximal} it there are no proper $\kappa$-cofinitary supergroups.  There is a long list of literature regarding the set theoretic properties of maximal cofinitary groups on $\omega$ (both from purely combinatorial, as well as descriptive set theoretic point of view), while higher counterparts have already been studied in~\cite{MADfamandNeighbors, MCGrevisited}. For more recent studies of a close relative to the $\kappa$-maximal cofinitary groups, namely $\kappa$-maximal almost disjoint families see~\cite{NoKappaMAD, KappaDefinableMAD}.

A main tool in our analysis of the structure of $\kappa$-maximal cofinitary groups is the following recent higher counterpart of Bell's theorem (see~\cite[Theorem 4.3]{SchilhanPkappa}): {\em If $\mathbb{P}$ is a $\kappa$-specifically centered partial order, $\kappa^{<\kappa}=\kappa$ and $\{D_\alpha:\alpha<\lambda\}$ is a family of dense subsets of $\mathbb{P}$, where $\lambda^{<\kappa}=\lambda$ and $\lambda<\mathfrak{p}(\kappa)$, then there is a filter $G\subseteq \mathbb{P}$ such that $G\cap D_\alpha\neq\emptyset$ for all $\alpha$.} Here $\mathfrak{p}(\kappa)$ denotes the $\kappa$-pseudointersection number.

Given a group $G \leq S(\kappa)$ we look at the orbits arising from its action on $\kappa$: for every $g \in G$ and $\alpha \in \kappa$ the action is such that $(g, \alpha) \mapsto g(\alpha)$. In \cite[Theorem 9]{Isomofmcgs} Kastermans showed that if $G$ is a maximal cofinitary group (on $\omega$) then $G$ has only finitely many orbits. Below, we obtain the following:

\begin{thm*}[See Theorem \ref{orbits1}]
If $G \leq S(\kappa)$ is a $\kappa$-maximal-cofinitary group then it has less than $\kappa$-many orbits.
\label{mainthm1}
\end{thm*}

Kastermans \cite[Theorem 10]{Isomofmcgs} also shows that in the $\omega$-case this is, at least consistently, the only restriction on the orbit structure of a maximal cofinitary group. Specifically he proves that under $\MA$, for any $n < \omega$ and $\{O_k \; | \; k < n\}$ is a partition of $\omega$ then there is a maximal cofinitary group $G$ whose orbits are exactly this partition. Building on the above mentioned generalization of Bell's theorem (see ~\cite{SchilhanPkappa}, as well as~\cite{TowersGapsUnctble}), we establish the following higher Baire spaces analogue: 

\begin{thm*}[See Theorem \ref{orbits2}]
If $\mathfrak{p}(\kappa) = 2^\kappa$ then given any $\lambda < \kappa$ and any partition $\{O_\alpha\; | \; \alpha < \lambda\}$ of $\kappa$ there is a $\kappa$-maximal cofinitary group $G$ so that the orbits of $G$ are exactly $\{O_\alpha \; | \; \alpha < \lambda\}$.
\label{mainthm2}
\end{thm*}

Finally we investigate the possible isomorphism types of $\kappa$-maximal cofinitary groups. Kastermans proved that under $\MA$ there is a maximal cofinitary group $G$ which is universal for groups of size less than continuum i.e. if $H$ is any group of size ${<}2^{\aleph_0}$ then $H$ embeds into $G$. Here we obtain the existence of 
an universal $\kappa$-maximal cofinitary group:

\begin{thm*}[See Theorem \ref{universaltheorem}]
Assume $\GCH$. For any regular $\delta > \kappa$ there is a forcing extension preserving cardinals and cofinalities in which $2^\kappa = \delta$ and there is a $\kappa$-maximal cofinitary group which is universal for groups of size ${<}\delta$.
\label{mainthm3}
\end{thm*}

If we restrict the conclusion of this theorem to groups of size ${\leq} \kappa$ then the statement follows from $\mfp(\kappa) = 2^\kappa$. In particular this conclusion follows from $2^\kappa = \kappa^+$. 

We conclude the paper with interesting remaining open questions.

\section{$\kappa$-Maximal Cofinitary Groups have ${<}\kappa$ Many Orbits}

In this section we prove Theorem \ref{mainthm1}. For convienence it's restated below.
\begin{theorem}
If $G \leq S(\kappa)$ is a $\kappa$-maximal cofinitary group then it has less than $\kappa$ many orbits.
\label{orbits1}
\end{theorem}

Note that the proof of this theorem follows along almost the exact same lines as Kastermans' proof, \cite[Theorem 9]{Isomofmcgs} of the $\omega$ case.

\begin{proof}
Assume $G \leq S(\kappa)$ is a $\kappa$-cofinitary group with $\kappa$ many orbits. Enumerate them without repetition as $\{O_\alpha \; | \; \alpha < \kappa\}$. We will show that $G$ is not maximal. In fact, we will produce a new permutation $h \notin G$ so that $\langle G, h\rangle := G * \langle h \rangle$ is $\kappa$-cofinitary. 

The construction of this $h$ is done in stages. We will recursively define ${<}\kappa$-length partial injections $h_\alpha : \kappa \to \kappa$ so that for all $\alpha$ we have $h_\alpha \subseteq h_{\alpha+1}$ and for limit ordinals $\lambda$ we have $h_\lambda = \bigcup_{\xi < \lambda} h_\xi$. The desired $h$ will be $\bigcup_{\alpha < \kappa} h_\alpha$. Let $h_0 = \emptyset$. Since we take unions at limit stages, we just need to define the successor stage. Assume $h_\alpha$ has been defined. Let $\xi = {\rm min} (\kappa \setminus {\rm dom }(h_\alpha) \cup \kappa \setminus {\rm range}(h_\alpha))$, let $\eta$ be least so that $({\rm dom }(h_\alpha) \cup {\rm range}(h_\alpha) \cup \{\xi\}) \cap O_\eta = \emptyset$ and let $\zeta = {\rm min}\,  O_\eta$. If $\xi \notin {\rm dom} (h_\alpha)$ let $h_{\alpha+1} = h_\alpha \cup \{(\xi, \zeta)\}$ and otherwise let $h_{\alpha+1} = h_\alpha \cup \{(\zeta, \xi)\}$. Let $h = \bigcup_{\alpha <\kappa} h_\alpha$. Clearly $h \in S(\kappa)$. Moreover, $h \notin G$ since, at each stage we moved something from one orbit to another since $\xi \notin O_\eta$ (for the same reason $h$ actually has no fixed points). It remains to show that $\langle G, h\rangle$ is a $\kappa$-cofinitary group. The following observation is important.

\begin{claim}
For each $\alpha, \beta < \kappa$ there is at most one pair $(\xi, \zeta) \in \kappa^2$ so that $\xi \in O_\alpha$, $\zeta \in O_\beta$ and $h(\xi) = \zeta$ or $h(\zeta) = \xi$. In other words, $h$ sends at most one element from $O_\alpha$ to an element of $O_\beta$ or vice versa for each pair of $\alpha$ and $\beta$.
\end{claim}

\begin{proof}[Proof of Claim]
Fix $\alpha, \beta < \kappa$ and suppose $h_\delta(\xi) = \zeta \in O_\beta$ with $\xi \in O_\alpha$ for some $\delta < \kappa$. Then, for all $\gamma \geq \delta$ it's not the case that ${\rm dom }(h_\gamma) \cup {\rm range}(h_\gamma) \cap O_\alpha = \emptyset$ and the same for $O_\beta$ so by the construction of $h$, no further pairs from $O_\alpha \times O_\beta$ or $O_\beta \times O_\alpha$ will be added to $h_\gamma$.
\end{proof}

Let $W_G (x)$ consist of the set of words or the form $w(x) = g_0 x^{k_0}g_1 x^{k_1}g_2...x^{k_{l-1}}g_l$ for $l < \omega$, $g_0, ..., g_l \in G$, $x$ a fresh variable symbol and $k_0, ..., k_{l-1} \in \mathbb Z$. Clearly every element of $\langle G, h\rangle$ can be represented as a word in $W_G(h)$. We want to show that for each $w(x) \in W_G(x)$ the word $w(h)$ is either the identity or has ${<}\kappa$ many fixed points. The key step in this is the introduction of the following tree.

\begin{definition}
The $G$-{\em orbits tree} of $h$ is the graph\footnote{In order to show the comparison between the current proof and Kastermans \cite[Definition 14]{Isomofmcgs} we keep his terminology of a $G$-orbits {\em tree}. However, in the our case, unlike in the case of $\omega$, this graph is {\em not} necessarily a tree, namely it's not necessarily connected. For example if, for all $n < \omega$ we let $O_{n} = \{n\}$ then we will have that $h(n) = n+1$ for all even finite $n$ and $h(n) = n-1$ for all odd finite $n$ and these orbits will form a connected component unto themselves (obviously this situation can't happen in the $\omega$ case). However the fact that the $G$-orbits tree is a tree and not just an acyclic graph is not used in Kastermans' proof and won't be needed in ours either.} consisting of vertex set $\{O_\alpha \; | \; \alpha < \kappa\}$ and for $\alpha \neq \beta < \kappa$ we put an edge from $O_\alpha$ to $O_\beta$ if there is an $\xi \in O_\alpha$ so that $h(\xi) \in O_\beta$ or vice versa.
\end{definition}

\begin{claim}
The $G$-orbits tree of $h$ is acyclic.
\end{claim}

\begin{proof}[Proof of Claim]
Suppose towards a contradiction that there are $\alpha_0,...,\alpha_{l+1} = \alpha_0$ so that for all $i <l+1$ there is an edge from $O_{\alpha_i}$ to $O_{\alpha_{i+1}}$. By the previous claim, for each $i < l+1$ there is a unique pair $(\xi, \zeta)$ so that $\xi \in O_{\alpha_i}$ and $\zeta \in O_{\alpha_{i+1}}$ and either $h(\xi) = \zeta$ or $h(\zeta) = \xi$. Let $\alpha$ be least so that all of these pairs are in $h_{\alpha+1}$. Since $\alpha$ is least, the pair $(\xi, \zeta) \in h_{\alpha + 1} \setminus h_\alpha$ is such that $\xi \in O_{\alpha_i}$ and $\zeta \in O_{\alpha_{i+1}}$ and either $h(\xi) = \zeta$ or $h(\zeta) = \xi$ for some $i < l+1$ and, moreover, there are already elements from both $O_{\alpha_i}$ and $O_{\alpha_{i+1}}$ in ${\rm dom}(h_\alpha) \cup {\rm range}(h_\alpha)$. But this contradicts the defining procedure for $h_{\alpha+1}$ so we get a contradiction.
\end{proof}

We're now ready to finish the proof. Suppose $w(x) \in W_G(x)$ and $w(h)$ is not the identity. Let $w(x) = g_0 x^{k_0}g_1 x^{k_1}g_2...x^{k_{l-1}}g_l$. Suppose that for some $\alpha$ we have that $w(h) (\alpha) = \alpha$. First let's see that this gives rise to a fixed point for some $g_j$. To see this, consider the sequence of ordinals $\alpha, g_l(\alpha), h^{{\rm sign}(k_{l-1})} (g_l(\alpha)), ..., w(h)(\alpha) = \alpha$ i.e. the sequence of ordinals that appear as we successively evaluate $w(h)(\alpha)$ from the right to the left. Call this the evaluation sequence. Consider also the sequence of orbits $O_{\alpha_0}, O_{\alpha_1}, ..., O_{\alpha_{l + |k_0| + ... + |k_{l-1}|}}$ in which each element of the evaluation sequence is in. Observe that, omitting repetitions that arise from the elements of $G$ appearing in $w(h)$, this sequence of orbits is a walk in the $G$-orbits tree of $H$. Moreover since $\alpha$ is a fixed point, it must be the case that this walk begins and ends in the same orbit and therefore there is an $i$ so that $O_{\alpha_i}$ is the furthest from $O_{\alpha_0}$ in the $G$-orbits tree of $h$. Fix the least such $i$. It follows that there are ordinals $\xi, \zeta$ in the evaluation sequence so that $\xi \in O_{\alpha_{i-1}}$, $\zeta \in O_{\alpha_i}$ and either $h(\xi) = \zeta$ or $h(\zeta) = \xi$. The cases are symmetric so assume the former. We have that since $O_{\alpha_i}$ maximizes the distance from $O_{\alpha_0}$, the next element in the sequence of orbits not equal to $O_{\alpha_i}$ must be closer, and, since the $G$-orbits tree of $h$ is acyclic, it must in fact be $O_{\alpha_{i-1}}$. By the shape of $w(x)$ we have that there is a piece of the word of the form $x^ig_j x$ with $i \in \{-1, 1\}$  and, that piece of the evlauation sequence is then $x(\xi) = \zeta$, $g_j(\zeta)$, $x^{i} (g_j(\zeta))$. Also, we know that $g_j(\zeta) \in O_{\alpha_i}$ since $g_j \in G$ and $O_{\alpha_i}$ is an orbit of $G$ and, by the argument above, $x^{i}( g_j(\zeta)) \in O_{\alpha_{i-1}}$. But by the first claim the only way this could occur is if $i = -1$ and $g_j(\zeta) = \zeta$. Therefore $\zeta$ is a fixed point of $g_j$.

To finish the proof now, note that if $w(h)$ has $\kappa$ many fixed points then, since there are only finitely many $g_j$'s there must be some $j$ so that $\kappa$ many of the fixed points of $w(h)$ give rise to a fixed point for the same $g_j$. Moreover, each one of these fixed points must give rise to a distinct fixed point in $g_j$ since these are all bijections and hence the corresponding evaluation sequences are never equal. But this means that some $g_j$ has $\kappa$ many fixed points contradicting the fact that it is a member of a $\kappa$-cofinitary group.

\end{proof}

A potentially interesting observation is that the above proof works not only for a $\kappa$-cofinitary group but in fact for a group $G \leq S(\kappa)$ so that every non-identity element has ${<}\mu$ fixed points for any fixed $\mu \leq \kappa$, in particular a group where the number of fixed points of each non-identity element is finite. This suggests that such $\mu$-cofinitary subgroups of $S(\kappa)$ may be interesting alternative analogues to the countable case on $\kappa$.

In \cite{Isomofmcgs} Kastermans sketches a corollary of the $\omega$ case of this result, due to Blass, that there is no Abelian maximal cofinitary group. It's clear by looking at the proof that the same idea works verbatim in our context. Hence we get the following.
\begin{corollary}[Blass, See Theorem 12 of \cite{Isomofmcgs}]
There is no Abelian $\kappa$-maximal cofinitary group.
\end{corollary}

\section{Controlling the Number of Orbits of a $\kappa$-Maximal Cofinitary Group}

In this section we prove Main Theorem \ref{mainthm2}. Towards this end we recall some facts about the cardinal $\mfp(\kappa)$ and a forcing notion introduced by the first author in \cite{MCGrevisited} for adding a $\kappa$-maximal cofinitary group of a fixed size. We refer the reader to \cite{TowersGapsUnctble} for more information on $\mfp(\kappa)$ and \cite{MCGrevisited} for more information about this forcing notion. 

\begin{definition}
Let $\mathcal F$ be a family of subsets of $\kappa$. We say that $\mathcal F$ has the {\em strong intersection property} if for any subfamily $\mathcal F' \subseteq \mathcal F$ of size ${<}\kappa$ the intersection $\bigcap \mathcal F '$ has size $\kappa$. The {\em pseudointersection number for} $\kappa$, denoted $\mfp(\kappa)$ is the least size of a family $\mathcal F$ with the strong intersection property for which there is no $A \in [\kappa]^\kappa$ which is almost contained\footnote{Here, if $B, C\in [\kappa]^\kappa$ we say that $B$ is {\em almost contained in} $C$ if $B \setminus C$ has size ${<}\kappa$.} in all $F \in \mathcal F$.
\end{definition}

We will not need this definition of $\mfp(\kappa)$ but rather an equivalent one in terms of a certain forcing axiom. To this end we give the following definition.

\begin{definition}
Let $\mathbb P$ be a ${<}\kappa$-closed, $\kappa$-centered poset with $\mathbb P = \bigcup_{\gamma < \kappa} C_\gamma$ witnessing $\kappa$-centeredness. We say that $\mathbb P$ has {\em canonical lower bounds} if there is a function $f = f^\mathbb P:\kappa^{<\kappa} \to \kappa$ such that whenever $\lambda < \kappa$ and $(p_\alpha \; | \; \alpha < \lambda)$ is a decreasing sequence of conditions with $p_\alpha \in C_{\gamma_\alpha}$ then there is a $p \leq p_\alpha$ for all $\alpha$ in $C_{f(\gamma_\alpha \; | \; \alpha < \lambda)}$.
\end{definition}

The significance of this definition is that the forcing axiom for the class of ${<}\kappa$-closed, $\kappa$-centered posets with canonical lower bounds is equivalent to $\mathfrak{p}(\kappa) = 2^\kappa$. More precisely,
\begin{theorem}[Theorem 1.8 of \cite{TowersGapsUnctble}]
If $\mathbb P$ is a ${<}\kappa$-closed, $\kappa$-centered poset with canonical lower bounds and below every $p \in \mathbb P$ there is a $\kappa$-sized antichain. Then for any collection $\mathcal D$ of ${<}\mathfrak{p}(\kappa)$ many dense subsets of $\mathbb P$, there is a filter $G \subseteq \mathbb P$ which meets every element of $\mathcal D$.
\label{forcingaxiom}
\end{theorem}

The point for us here is that we will show the forcing notion introduced in \cite{MCGrevisited} is a ${<}\kappa$-closed, $\kappa$-centered poset with canonical lower bounds and hence, under the hypothesis $\mfp(\kappa) = 2^\kappa$ we have a forcing axiom for this poset which is what we will need to show Main Theorem \ref{mainthm2}. We now recall the relevant poset.

Let $A$ be an index set and denote by $W_A$ the set of all reduced words in the alphabet $\{a^i \; | \;a\in A \; {\rm and} \; i \in \{-1, 1\}\}$. We denote by $\widehat{W}_A$ the {\em good words} i.e. those that are either the power of a singleton or start and end with a different letter. Given a mapping $\rho: A \to S(\kappa)$, denote by $\hat{\rho}$ its canonical extension to a group homomorphism from the free group $F_A$ on $A$ to $S(\kappa)$. We say that $\rho$ {\em induces a} $\kappa$-{\em cofinitary representation} if the image of $\hat{\rho}$ is $\kappa$-cofinitary. Whenever $s \subseteq A \times \kappa \times \kappa$ we let $s_a = \{(\alpha, \beta) \in \kappa \times \kappa \; | \; (a, \alpha, \beta) \in s\}$ and $s^{-1}_a = \{(\beta, \alpha) \; | \; (a, \alpha, \beta) \in s\}$ for each $a$ in $A$. If each $s_a$ is a partial function we may occasionally abuse notation and write $s_a(\alpha)$ for the unique $\beta$ so that $(a, \alpha, \beta) \in s$ (if it exists). For a word $w \in W_A$ and a set $s \subseteq A \times \kappa \times \kappa$ recall the relation $e_w[s] \subseteq \kappa \times \kappa$ defined recursively by stipulating

\begin{enumerate}
\item
if $w = a$ for some $a \in A$ then $e_w[s] = s_a$.
\item
if $w = a^{-1}$ for some $a \in A$ then $e_w[s] = s_a^{-1}$.
\item
if $w = a^iu$ for some $u \in W_A$, $a \in A$ and $i \in \{-1, 1\}$ without cancelation then $(\alpha, \beta) \in e_w[s]$ if and only if there is a $\gamma \in \kappa$ so that $(\gamma, \beta) \in e_{a^i}[s]$ and $(\beta, \alpha) \in e_u[s](\alpha, \gamma)$.
\end{enumerate}
We also define $e_w[s, \rho]$ for $\rho:B \to S(\kappa)$ inducing a $\kappa$-cofinitary representation by $(\alpha, \beta) \in e_w[s, \rho]$ if and only if $(\alpha, \beta) \in e_w[s \cup \{(b, \gamma, \delta) \: | \; b \in B \; {\rm and} \; \rho(b)(\gamma) = \delta\}]$.

\begin{definition}[See \cite{MCGrevisited}, Definition 2.2]
Let $A$ and $B$ be disjoint sets and let $\rho:B \to S(\kappa)$ induce a $\kappa$-cofinitary representation. The forcing notion $\mathbb Q = \mathbb Q^\kappa_{A, \rho}$ is defined as the set of all pairs $(s, F) \in [A \times \kappa \times \kappa]^{<\kappa} \times [\widehat{W}_{A \cup B}]^{<\kappa}$ so that $s_a$ is injective for each $a \in A$. We let $(s, F) \leq_\mathbb Q (t, E)$ if $s \supseteq t$, $F \supseteq E$ and for every $\alpha \in \kappa$ and $w \in E$ if $e_w[s, \rho](\alpha) = \alpha$ then $e_w[t, \rho](\alpha) = \alpha$ (and in particular is defined).
\label{poset1}
\end{definition}

It's clear that $\mathbb Q^\kappa_{A, \rho}$ is ${<}\kappa$-closed and it's shown in \cite[Lemma 2.3]{MCGrevisited} that $\mathbb Q^\kappa_{A, \rho}$ is $\kappa^+$-Knaster for any $A$ and $\rho$. We want to show that $\mathbb Q^\kappa_{A, \rho}$ has canonical lower bounds when $|A| \leq \kappa$. First observe the following.

\begin{proposition}
If $|A| \leq \kappa$ then $\mathbb Q^\kappa_{A, \rho}$ is $\kappa$-centered.
\end{proposition}

\begin{proof}
It suffices to note that for {\em any} $A$ and $\rho$, any two conditions in $\mathbb Q^\kappa_{A, \rho}$ with the same first coordinate are compatible. Thus, if $|A| \leq \kappa$ then there are only $\kappa$ many first coordinates so the sets $C_s = \{(t, F) \in \mathbb Q \; | \; t = s\}$ witness the $\kappa$-centeredness of $\mathbb Q$.
\end{proof}

\begin{lemma}
If $|A| \leq \kappa$ then $\mathbb Q^\kappa_{A, \rho}$ has canonical lower bounds.
\end{lemma}

\begin{proof}
Fix bijections $g_0:\kappa^{<\kappa} \to [A \times \kappa \times \kappa]^{<\kappa}$ and $g_1:\kappa \to [A \times \kappa \times \kappa]^{<\kappa}$. For each $\gamma < \kappa$ let $C_\gamma$ be the set of conditions with first coordinate $g_1 (\gamma)$. As mentioned above these sets are all directed. Let $f:\kappa^{<\kappa} \to \kappa$ be defined by, for each $\lambda < \kappa$ and $(\lambda_\alpha \; | \; \alpha < \lambda) \in \kappa^{<\kappa}$, setting $f(\lambda_\alpha \; | \; \alpha < \lambda) = g_1^{-1}({\rm sup} ( g_0(\lambda_\alpha \: \ \; \alpha < \lambda)))$. In other words, the lower bound of any sequence of conditions $(p_\alpha \; | \; \alpha < \lambda)$ can be found by taking the sup of the first coordinates and looking in the directed set associated with that supremum.
\end{proof}

As a consequence of this lemma we have the following.

\begin{lemma}
Assume that $B$ is a set of size less than $\mathfrak{p}(\kappa)$ and $\rho:B \to S(\kappa)$ induces a $\kappa$-cofinitary representation. Then there is an $h \in S(\kappa) \setminus \hat{\rho}$ so that the group freely generated by the image of $\hat{\rho}$ and $h$ is $\kappa$-cofinitary.
\label{addanewelmt}
\end{lemma}

We also get the following.
\begin{theorem}
$\mfp(\kappa) \leq \mathfrak{a}_g(\kappa)$
\end{theorem}
This was known, though it does not seem to have appeared in the literature in this explicit form. Instead it has been shown that $\mfp(\kappa) \leq \mfb(\kappa)$ (see \cite[Theorem 2.9]{SchilhanPkappa}) and that $\mfb(\kappa) \leq \mathfrak{a}_g(\kappa)$ (see \cite[Theorem 3.2]{MADfamandNeighbors}).

Having set up the necessary facts about $\mathbb Q^\kappa_{A, \rho}$ we are now ready to begin proving Main Theorem \ref{mainthm2}. For convenience we recall it below.

\begin{theorem}
If $\mathfrak{p}(\kappa) = 2^\kappa$ then given any $\lambda < \kappa$ and any partition $\{O_\alpha\; | \; \alpha < \lambda\}$ of $\kappa$ there is a $\kappa$-maximal cofinitary group so that $\mathsf{ORB}(G) = \{O_\alpha \; | \; \alpha < \lambda\}$.
\label{orbits2}
\end{theorem}

The rest of the section is devoted to proving Theorem \ref{orbits2}. As in the case of Theorem \ref{orbits1} we follow the corresponding proof of \cite[Theorem 16]{Isomofmcgs}. From now on assume $\mathfrak{p}(\kappa) = 2^\kappa$. This assumption will be used via Theorem \ref{forcingaxiom}. Given a group $G \leq S(\kappa)$ denote by $\mathsf{ORB}(G)$ the set of orbits of $G$. Fix $\lambda, \mu < \kappa$ and partition $\kappa$ into $\lambda$ many bounded pieces $\{B_\alpha \; | \; \alpha < \lambda\}$ and $\mu$ many unbounded pieces $\{U_\gamma \; | \; \gamma < \mu\}$. To begin with fix a group $G_\kappa$, so that $\mathsf{ORB}(G_\kappa) = \{B_\alpha \; | \; \alpha < \lambda\} \cup \{U_\gamma \; | \; \gamma < \mu\}$ and $G_\kappa$ is isomorphic to the free group with $\kappa$ many generators. To see that such a group exists note that,for any $\nu \leq \kappa$ using the forcing notion $\mathbb Q^\nu_{A, \rho}$ and the fact that is is ${<}\nu$-closed, we can find $\nu$-many elements of $S(\nu)$ which generate a $\nu$-cofinitary group which is free and acts transitively on $\nu$. Now, pushing forward these groups onto each $B_\alpha$ and $U_\gamma$ via a bijection with its cardinal we can find a group $G$ of size $\kappa$ which is freely generated and for each $g \in G$ we have that $g \hook U_\gamma$ is an element of the free group we built acting transitively on $U_\gamma$ and idem for $B_\gamma$.

\begin{remark}
There is a subtle difference here between the uncountable and the countable case. Namely, in the countable case we need only one element to generate the type of group described above whereas a counting argument shows that in the uncountable case we need $\kappa$ many.
\end{remark}

Now, enumerate $S(\kappa)$ as $\{f_\alpha \; | \; \kappa < \alpha < 2^\kappa\}$. We will build a continuous, increasing chain of $\kappa$-cofinitary groups $G_\alpha$ for $\kappa \leq \alpha < 2^\kappa$ so that for each $\alpha$ we have $\mathsf{ORB}(G_\alpha) =  \{B_\gamma \; | \; \gamma < \lambda\} \cup \{U_\gamma \; | \; \gamma < \mu\}$ and either $f_\alpha \in G_\alpha$ or $\langle G_{\alpha +1}, f_\alpha\rangle$ is not $\kappa$-cofinitary. Clearly this will suffice to prove the theorem. To begin we need several lemmas. The first was proved by the first author in the original paper on $\kappa$-cofinitary groups.

\begin{lemma}[Lemma 2.6 of \cite{MCGrevisited}]
Let $(s, F)\in \mathbb Q^\kappa_{A, \rho}$, $a \in A$.
\begin{enumerate}
\item
Let $\alpha \in \kappa \setminus {\rm dom}(s_a)$. Then there is an $I = I_{a, \alpha}$ such that $|\kappa \setminus I| < \kappa$ and for all $\beta \in I$ we have that $(s \cup \{(a, \alpha, \beta)\}, F) \leq (s, F)$. 
\item
Let $\beta \in \kappa \setminus {\rm ran}(s_a)$. Then there is a $J = J_{a, \beta}$ such that $|\kappa \setminus J| < \kappa$ and for all $\alpha \in J$ we have that $(s \cup \{ ( a, \alpha, \beta ) \}, F) \leq (s, F)$. 
\end{enumerate}
\label{extensions}
\end{lemma}

As an immediate consequence of this we get the following result which will be used in the proof of Theorem \ref{orbits2}. Note that the following is a slight strengthening of Lemma \ref{addanewelmt}.

\begin{lemma}
Assume that $B$ is a set of size less than $\mathfrak{p}(\kappa)$, $T \in [\kappa]^\kappa$ and $\rho:B \to S(\kappa)$ induces a $\kappa$-cofinitary representation. Then there is an $h \in S(\kappa) \setminus \hat{\rho}$ so that the group freely generated by the image of $\hat{\rho}$ and $h$ is $\kappa$-cofinitary and $|h \cap T \times T| = \kappa$.
\label{addanewelmt2}
\end{lemma}

\begin{proof}
Let $A = \{a\}$ be a singleton. Fix $(s, F) \in \mathbb Q^\kappa_{A, \rho}$ and $\alpha < \kappa$. Since $T$ is unbounded there is a $\beta > \alpha$ so that $\beta \in T$. Now by Lemma \ref{extensions} there are co-$\kappa$ many $\gamma$ so that $(s \cup \{a, \beta, \gamma\}, F) \in \mathbb Q^\kappa_{A, \rho}$, in particular such a $\gamma$ can be found in $T$. In other words, the set of conditions $(s, F)$ so that $s_a \cap (A \setminus \alpha) \times A \neq \emptyset$ is dense in $\mathbb Q^\kappa_{A, \rho}$. Applying Theorem \ref{forcingaxiom} to this collection of dense sets alongside those used in Lemma \ref{addanewelmt} then gives the requisite $h$.
\end{proof}

We also need a generalization of the notion of a ``hitable" function from Kastermans' proof. 

\begin{definition}
Let $C \in [\kappa]^\kappa$, $G \leq S(C)$ and let $f:C \to C$ be a partial injection of size $\kappa$. We say that $f$ is {\em hitable} with respect to $G$ if for every $g \in G$ we have $|f \setminus g| = \kappa$ and for each $w \in W_G(x)$ we have that either $w(f)$ is the identity or has only ${<}\kappa$ many fixed points (on the domain in which it's defined). 
\end{definition}

Note that if $C = {\rm dom}(f)$ then being hitable means that $f \notin G$ and $\langle G, f\rangle$ is $\kappa$-cofinitary. The point of this definition is that we need to construct our groups in such a way that if $f_\alpha$ is hitable with respect to $G_\alpha$ but $\langle G_\alpha, f_\alpha\rangle$ collapses orbits then we need a way to ``kill" the fact that $f_\alpha$ is hitable without changing the orbits of $G_\alpha$. To explain this more succinctly, for $G \leq H \leq S(\kappa)$ let us say that $H$ {\em preserves the orbit structure of} $G$ if $\mathsf{ORB}(G) = \mathsf{ORB}(H)$.

\begin{lemma}
Let $C \in [\kappa]^\kappa$ and $f:C \to C$ be a partial injection of size $\kappa$. Suppose that $G \leq S(\kappa)$ is a $\kappa$-cofinitary group induced by a mapping $\rho:B \to S(\kappa)$, $A$ is a set disjoint from $B$ and $f$ is hitable with respect to $G$. Let $(s, F) \in \mathbb Q^\kappa_{A, G}$ and $a \in A$. Then there is an $\alpha \in {\rm dom}(f)$ so that $(s \cup \{(a, \alpha, f(\alpha))\}, F) \leq (s, F)$.
\label{hittinglemma}
\end{lemma}

\begin{proof}
It suffices to prove the lemma in the case that $|F| = 1$ since we can iterate the argument ${<}\kappa$ many times to generalize it. Suppose that $F = \{w\}$. Fix $a \in A$. To begin, observe that we may assume that ${\rm dom}(f) \cap {\rm dom}(s_a) = \emptyset$ since $|{\rm dom}(s_a) | < \kappa$ and $|{\rm dom}(f)| = \kappa$ so we can ``thin out" $f$ to a hitable function with domain disjoint from $s_a$. Similarly we can assume that $f$ itself has no fixed points by thinning out. 

We will show that there are $\kappa$ many $\alpha \in {\rm dom}(f)$ so that the fixed points of $e_w[s \cup \{(a, \alpha, f(\alpha)\}, \rho]$ are the same as the fixed points of $e_w[s, \rho]$. Clearly this suffices to prove the lemma. There are several cases.

\noindent \underline{Case 1}: There is no occurrence of $a$ in $w$. In this case $e_w[s, \rho] = e_w[s \cup \{(a, \alpha, f(\alpha))\}, \rho]$ so there are no new fixed points.

\noindent \underline{Case 2}: There are both occurrences in $w$ of $a$ or $a^{-1}$ and of some other $b \neq a$ or $b^{-1}$ with $b \in A$. In this case, since $|s_b| < \kappa$ and $|{\rm dom}(f)| = \kappa$, there are $\kappa$ many points in the domain of $f$ so that for any $\gamma < \kappa$ if $e_w[s \cup \{(a, \alpha, f(\alpha))\}, \rho] (\gamma)$ is defined then already $e_w[s, \rho] (\gamma)$ is defined. Therefore again $e_w[s, \rho] = e_w[s \cup \{(a, \alpha, f(\alpha))\}, \rho]$

\noindent \underline{Case 3}: Only $a$, $a^{-1}$ and elements of $B$ occur in $W$. In this case there are two further subcases. To explain, let $e_w[f, \rho]$ denote the partial permutation obtained by replacing every instance of $s_a$ in the evaluation of $w$ with $f$. In essence, this is akin to treating the word $w$ as a word in ${\rm Im}(\hat{\rho})$ with a free variable and substituting in $f$ for this free variable. Since $f$ is hitable with respect to ${\rm Im}(\hat{\rho})$, it must be the case that $e_w[f, \rho]$ is either the identity on its domain or is $\kappa$-cofinite and these are the two subcases. 

If $e_w[f, \rho]$ is $\kappa$-cofinite then there is a $\kappa$ sized subset of $\alpha$ in the domain of ${\rm dom}(f)$ so that the pair $(\alpha, f(\alpha))$ does not lead to any fixed points in $e_w[f, \rho]$ and therefore (recalling that the domains of $s_a$ and $f$ are disjoint by assumption), $e_w[s \cup \{(a, \alpha, f(\alpha))\}, \rho]$ has the same fixed points as $e_w[s, \rho]$ for any such $\alpha$.

If $e_w[f, \rho]$ is the identity on its domain, then it must be the case that there are two occurrences of $a$ or $a^{-1}$ in the word $w$. Thus there is either an occurrence of $a^2$, $a^{-2}$ or $a^ib_0...b_ka^j$ for $b_0,...,b_k \in B$ and $i, j \in \{-1, 1\}$. In any of these cases note that there are ${<}\kappa$ many $\gamma$ so that the same $(\alpha, f(\alpha))$ is
used in the evaluation of $e_w[f, \rho](\gamma)$. Therefore we can find a $\kappa$-sized subset of $\alpha$ in the ${\rm dom}(f)$ so that $(\alpha, f(\alpha))$ is only used in one of the two occurrences and thus for any such $\alpha$ we have $e_w[s \cup \{(a, \alpha, f(\alpha))\}, \rho]$ will not contain any new elements so once again $e_w[s, \rho] = e_w[s \cup \{(a, \alpha, f(\alpha))\}, \rho]$ so there are no new fixed points.

Since this concludes all of the cases the proof is not complete.
\end{proof}

\begin{lemma}
Suppose $G \leq S(\kappa)$ is a $\kappa$-cofinitary group of size ${<}\mathfrak{p}(\kappa)$, $C \in [\kappa]^\kappa$ an $f:C \to C$ is an injection which is hitable with respect to $G$. Then there is a permutation $g \in S(\kappa) \setminus G$ so that $\langle G, g\rangle$ is $\kappa$-cofinitary and $|g \cap f|= \kappa$.
\label{addahit}
\end{lemma} 

\begin{proof}
By Lemma \ref{hittinglemma}, for each $\alpha < \kappa$ and $a \in A$ there is a dense set of conditions $(s, F)$ in $\mathbb Q^\kappa_{A, \rho}$ so that there is a $\beta > \alpha$ and $(a, \beta, f(\beta)) \in s$. Given this the construction of $g$ follows similarly to that in Lemmas \ref{addanewelmt} and \ref{addanewelmt2}.
\end{proof}

We can now prove Theorem \ref{orbits2}.

\begin{proof}[Proof of Theorem \ref{orbits2}]
As mentioned before we will construct a continuous increasing sequence of $\kappa$-cofinitary free groups $G_\alpha$ for $\kappa \leq \alpha < 2^\kappa$. The desired group will be $G = \bigcup_{\kappa \leq \alpha <2^\kappa} G_\alpha$. We already stated what $G_\kappa$ is and for $\kappa < \lambda < 2^\kappa$ limit we have that $G_\lambda = \bigcup_{\gamma < \lambda} G_\gamma$. Fix $\alpha$ and assume that we have constructed $G_\alpha$ and that we have $\rho_\alpha: \alpha \to S(\kappa)$ induces a $\kappa$-cofinitary representation with ${\rm Im}(\hat{\rho}_\alpha) = G_\alpha$. We need to find a new permutation $h:\kappa \to \kappa$ so that $\langle G_\alpha , h\rangle$ is $\kappa$-cofinitary, $\langle G_\alpha , h\rangle \cong G_\alpha * \mathbb Z$ and if $f_\alpha \notin G_\alpha$ then $\langle G_\alpha, h, f_\alpha\rangle$ is not $\kappa$-cofinitary. We will define $h$ separately on each $U_\gamma$ and $B_\gamma$ in such a way that $h\hook U_\beta: U_\beta \to U_\beta$ is a bijection for each $\beta$ and similarly for each $B_\beta$.

For each $B_\beta$ let $h\hook B_\beta = g \hook B_\beta$ for some $g \in G_\kappa$. Note that since there are less than $\kappa$ many bounded subsets $B_\gamma$, we don't need to be worried if $w(h)  \hook \bigcup_{\gamma < \lambda} B_\gamma$ has fixed points for any $w \in W_\alpha$. It suffices to ensure that for each $U_\gamma$ and each word $w \in W_\alpha$ the permutation $w(h) \hook U_\gamma: U_\gamma \to U_\gamma$ has ${<}\kappa$ many fixed points\footnote{Note this is where the fact that we have ${<}\kappa$ many orbits is used and it's exactly this that would break if we tried to rerun the proof with $\kappa$ many orbits.}.

We will construct $h\hook U_\gamma$ differently depending on $f_\alpha$. 

\noindent \underline{Case 1}: $f_\alpha \in G_\alpha$ or $\langle G_\alpha, f_\alpha \rangle$ is not cofinitary. In this case we don't need to worry about $f_\alpha$ and so we construct $h \hook U_\gamma$ for each $\gamma$ using Lemma \ref{addanewelmt}.

\noindent \underline{Case 2}: $f_\alpha \notin G_\alpha$, $\langle G_\alpha, f_\alpha \rangle$ is cofinitary and there a $\gamma$ and an unbounded subset $U \subseteq U_\gamma$ so that $f_\alpha \hook U:U \to U$. Fix such a $\gamma$ and $U$. This means that $f_\alpha \hook U$ is hitable, and, on this unbounded subset respects the orbit $U_\gamma$. Therefore by Lemma \ref{addahit} we can construct $h\hook U_\gamma: U_\gamma \to U_\gamma$ so that $|h \hook U \cap f\hook U| = \kappa$. Construct the remaining $h \hook U_\xi$ as in Case 1. Note that in this case now $\langle G_\alpha, h, f_\alpha\rangle$ will no longer be cofinitary since $f_\alpha^{-1} \circ h \hook U$ will have $\kappa$ many fixed points.

\noindent \underline{Case 3}: The first two cases fail. In particular, $f_\alpha \notin G_\alpha$, $\langle G_\alpha, f_\alpha \rangle$ is cofinitary and there are distinct $ \gamma, \xi < \mu$ so that $|f_\alpha \cap U_\gamma \times U_\xi| = \kappa$. Fix such $\gamma$ and $\xi$ and let $U \subseteq U_\gamma$ be unbounded so that ${\rm Im}(f \hook U) \subseteq U_\xi$. We construct $h\hook U_\xi$ and $h\hook U_\gamma$ as follows. First use Lemma \ref{addanewelmt2} to construct $h \hook U_\xi$ so that $h \hook U_\xi$ moves $\kappa$ many elements of ${\rm Im}(f_\alpha)$ to itself. Now consider the function $f_\alpha^{-1} \circ (h \hook U_\xi) \circ (f_\alpha \hook U)$. This is a partial function from $U_\gamma$ to $U_\gamma$ with unbounded domain. If it is hitable with respect to this unbounded domain then follow the same procedure as in Case 2 to make sure that $h\hook U_\gamma$ and $f_\alpha^{-1} \circ (h \hook U_\xi) \circ (f_\alpha \hook U)$ intersect on a set of size $\kappa$. If $f_\alpha^{-1} \circ (h \hook U_\xi) \circ (f_\alpha \hook U)$ is not hitable on any unbounded set then already we have ensured that $f_\alpha$ cannot be added to $G_{\alpha + 1}$. In either case now extend $h$ to all $U_\zeta$ as in Case 1.

Finally let $\rho_{\alpha + 1}:\alpha + 1 \to S(\kappa)$ extend $\rho_\alpha$ by stipulating that $\rho_{\alpha+1} (\alpha) = h$. By what we have shown we get that $G_{\alpha+1} : = {\rm Im}(\hat{\rho}_{\alpha + 1}) \supseteq G_\alpha$ is a $\kappa$-cofinitary group respecting the orbits of $G_\alpha$ so that if $f_\alpha \notin G_\alpha$ then $\langle G_{\alpha+1}, f_\alpha \rangle$ is not $\kappa$-cofinitary. This completes the construction.

Let $G = \bigcup_{\alpha < 2^\kappa} G_\alpha$. Clearly this is a $\kappa$-cofinitary group whose orbits are exactly $\{B_\alpha \; | \; \alpha < \lambda\}$ and $\{U_\gamma \; | \; \gamma < \mu\}$. We need to check that $G$ is maximal. Suppose $f = f_\alpha \in S (\kappa)$ is a permutation so that $f_{\alpha} \notin G$. By our construction that means that the element $h$ added at stage $\alpha$ was such that either $f_\alpha \circ h^{-1}$ or $f_\alpha^{-1} \circ h \circ f_\alpha$ has $\kappa$ many fixed points. In either case this means that $\langle G, f_\alpha\rangle$ is not $\kappa$-cofinitary thus completing the proof.
\end{proof}

\section{Universal $\kappa$-Maximal Cofinitary Groups}

In this section we prove Main Theorem \ref{mainthm3}, which generalizes \cite[Theorem III.15]{KastermansPhD}\footnote{Kastermans actually works under the assumption that $\CH$ holds as opposed to $\MA$ and just shows that the group embeds all countable groups. However, it's obvious how to generalize his construction to $\MA$ and ${<}2^{\aleph_0}$.}. The main ingredient is a slight augmentation of the forcing notion $\mathbb Q^\kappa_{A, \rho}$ to one which extends the image of $\rho$ to a $\kappa$-cofinitary group which embeds an arbitrary group $H$ as opposed to just the free group $F_A$. Let $\rho:B \to S(\kappa)$ be a mapping inducing a $\kappa$-cofinitary representation with $H \cap B = \emptyset$ (as sets). Considering $H$ as a set, let $W_H$ be the words in $\{a^{i} \; | \; a \in H \; i \in \{-1, 1\}\}$. For a word $w \in W_H$ we write $w \cong 1$ if by applying the rules $aa^{-1} = a^{-1}a = 1$ and all cancellations that result from the operation of $H$ gives the identity. 
\begin{definition}
Let $H$ be a group, $B$ a set disjoint from $H$ and $\rho:B \to S(\kappa)$ a mapping inducing a $\kappa$-cofinitary representation. The forcing notion $\mathbb Q^\kappa_{H, \rho}$ consists of pairs $(s, F) \in [H \times \kappa \times \kappa]^{<\kappa} \times [\widehat{W}_{H \cup B}]^{<\kappa}$ so that $(s, F)$ satisfies the same requirements as in Definition \ref{poset1} with $A = H$ (as a set) and if $w \in F$ with $w \cong 1$ then $e_W[s, \rho]  = id$ (on its domain). The extension relation is the same as in Definition \ref{poset1}.
\end{definition}

Essentially the same argument as in \cite{MCGrevisited} shows that this forcing notion is ${<}\kappa$-closed and $\kappa^+$-Knaster. Moreover, the same argument as in the previous section shows that if $|H| \leq \kappa$ then this poset is ${<}\kappa$-centered with canonical lower bounds. We will first show the following.
\begin{theorem}
Let $H$, $B$ and $\rho$ be as above. Then the forcing notion $\mathbb Q^\kappa_{H, \rho}$ forces that $H$ embeds into $S(\kappa)$. In fact the image of the generic embedding is isomorphic to ${\rm Im}(\hat{\rho}) * H$, and this group is $\kappa$-cofinitary.
\label{embeddingthm}
\end{theorem}

The key to proving theorem \ref{embeddingthm} is the following idea of {\em applying relations} from $H$ restricted to some small subset. To avoid unnecessary repetitions, for the next few lemmas let us fix a group $H$ and a $\rho:B \to S(\kappa)$. Write $\mathbb Q$ for $\mathbb Q^\kappa_{H, \rho}$.

\begin{definition}
Let $A \subseteq H$ be of size ${<} \kappa$ and $(s, F) \in \mathbb Q$. We say that $(t, F) \in [H \times \kappa \times \kappa]^{<\kappa} \times [\widehat{W}_{H \cup B}]^{<\kappa}$ is {\em obtained from} $(s, F)$ {\em by applying} $A$-{\em relations} if for every $a \in A$ we have $(a, \alpha, \beta) \in t$ if and only if there is a $w \in W_H$ so that $aw \cong 1$ and $e_w[s, \rho](\beta) = \alpha$. 
\end{definition}

Informally $(t, F)$ is obtained from $(s, F)$ by applying $A$-relations if $t$ extends $s$ to be closed under ``everything it has to be" vis-\`{a}-vis words whose letters come from $A$. Since $|A| < \kappa$ it follows that $|t| < \kappa$. Moreover it's immediate that $s \subseteq t$. In fact $(t, F) \in \mathbb Q$ and $(t, F) \leq (s, F)$. To see this we need the following lemma.
\begin{lemma}
Suppose that $A \subseteq H$ is of size ${<} \kappa$, $(s, F) \in \mathbb Q$ and $(t, F)$ is obtained from $(s, F)$ by applying $A$-relations. Let $w \in W_H$ and suppose that $e_w[t, \rho](\alpha) = \beta$. Then there is a word $\bar{w} \in W_H$ so that $w$ and $\bar{w}$ reduce to the same element of $H$ and $e_{\bar{w}}[s, \rho] (\alpha) = \beta$. 
\label{replacementlemma}
\end{lemma}

\begin{proof}
Suppose that $e_w[s, \rho](\alpha)$ is not defined but $e_w[t, \rho](\alpha) = \beta$. It follows that there is an $a \in A$ so that $w = u a v$ and a pair $(\gamma, \xi)$ so that $(a, \gamma, \xi) \in t$ and hence, by definition, there is a word $w_0$ so that $a w_0 \cong 1$ and $e_{w'}[s, \rho] (\xi) = \gamma$. Thus we can omit the occurrence of $a$ in $w$ by replacing it with $w_0^{-1}$. Continuing to apply this procedure we can reduce $w$ to a new word $\bar{w}$ so that $w \cong \bar{w}$ (as words in $H$) and $e_{\bar{w}}[s, \rho]$ is defined.
\end{proof}

As a corollary of this lemma we get the following.
\begin{lemma}
Suppose that $A \subseteq H$ is of size ${<} \kappa$, $(s, F) \in \mathbb Q$ and $(t, F)$ is obtained from $(s, F)$ by applying $A$-relations. The following hold. 
\begin{enumerate}
\item
$(t, F) \in \mathbb Q$
\item
$(t, F) \leq_\mathbb Q (s, F)$
\item
If $(t', F)$ is obtained from $(t, F)$ by applying $A$-relations then $t' = t$.
\end{enumerate}
\end{lemma}

\begin{proof}
The only one that's not obvious in light of Lemma \ref{replacementlemma} is the third one. Suppose that there is a $(a, \alpha, \beta) \notin t$ so that $a \in A$ there is a $w \in W_H$ so that $aw \cong 1$ and $e_w[t, \rho](\beta) = \alpha$. By applying the procedure described in Lemma \ref{replacementlemma} we can find a $\bar{w}$ so that $a\bar{w} \cong 1$ and $e_{\bar{w}}[s, \rho](\beta) = \alpha$. As a result $(a, \alpha, \beta) \in t$ contradicting our assumption.
\end{proof}

With these lemmas proved, we can now show Theorem \ref{embeddingthm}.

\begin{proof}[Proof of Theorem \ref{embeddingthm}]
The main thrust of what's left to do consists in showing that for each $h \in H$, $w \in W_{H \cup B}$ and $\alpha < \kappa$ the sets $D_{h, \alpha} = \{(s, F) \in \mathbb Q \; | \; \exists \beta \, (h, \alpha, \beta) \in s\}$, $R_{h, \alpha}= \{(s, F) \in \mathbb Q \; | \; \exists \beta \, (h, \beta, \alpha) \in s\}$ and $W_w = \{(s, F) \in \mathbb Q\; | \; w \in F\}$ are dense. That $W_w$ is dense is clear and the proofs for the other two are the same so we only prove that $D_{h, \alpha}$ is dense. 

Fix a condition $(s, F) \in \mathbb Q$ and apply $\{h\}$-relations to it to obtain a $(t, F) \leq (s, F)$. If $(t, F) \in D_{h, \alpha}$ we're done so suppose not. By Lemma \ref{extensions} there are co-boundedly many $\beta$ so that $(t \cup \{h, \alpha, \beta\}, F)$ is a condition in the bigger forcing notion defined by Definition \ref{poset1}. We just need to show that for some such $\beta$ $(t \cup \{h, \alpha, \beta\}, F)$ is a condition in $\mathbb Q$. From this it will follow immediately that $(t \cup \{h, \alpha, \beta\}, F) \leq _\mathbb Q (s, F)$ since the extension relation is the same. Let $\beta > {\rm dom}(t_h) \cup {\rm range}(t_h) \cup \{\alpha\}$. We claim that this $\beta$ works. Indeed suppose not and $w$ be a word of minimal length so that $e_w[t \cup \{h, \alpha, \beta\}, \rho] \neq id$ but $w \cong 1$. Fix a $\gamma$ so that $e_w[t \cup \{h, \alpha, \beta\}, \rho] (\gamma)\neq \gamma$. Since $\beta$ is larger than anything in the domain or range of $w$ it must be the case that the pair $(\alpha, \beta)$ is used in the evaluation of $e_w[t \cup \{h, \alpha, \beta\}, \rho](\gamma)$ either in the beginning of $w$, at the end of $w$ or in the middle, in which case it needs to be used again in reverse immediately. The third case would imply that $w$ is not minimal so this cannot be true. Therefore $w$ can either be written as $w'h^i$ or $h^iw'$ for $i \in \{-1, 1\}$ and $w' \in W_H$. In either case it follows that since $t$ was  obtained by applying $\{h\}$-relations $(h, \alpha, \beta) \in t$ which is a contradiction to our assumption.

To finish the proof now observe that by these density arguments $\mathbb Q$ adds an injective map from $H$ into $S(\kappa)$ whose image is $\kappa$-cofinitary. That it is in fact an embedding follows from the fact that closing under $\{h\}$-relations is dense for each $\{h\}$ and therefore the image of the embedding is isomorphic to ${\rm Im}(\hat{\rho}) * H$.
\end{proof}

Using Theorem \ref{embeddingthm} we can now prove Main Theorem \ref{mainthm3}. We recall it below for convenience.

\begin{theorem}
Assume $\GCH$. For any regular $\delta > \kappa$, there is a ${<}\kappa$-closed, $\kappa^+$-c.c. forcing notion $\mathbb P_\delta$ forcing that there is a $\kappa$-maximal cofinitary group $G$ of size $\delta$ which embeds every group of size ${<}\delta$. 
\label{universaltheorem}
\end{theorem}

\begin{proof}
Assume $\GCH$ and fix $\delta > \kappa$ regular. Obviously the desired poset will be an iteration whose iterands will be of the form $\mathbb Q^\kappa_{H, \rho}$. Let us fix a bookkeeping function $F:\delta \to [\delta]^{<\delta}$, which is surjective and so that for each $a \in [\delta]^{<\delta}$ the preimage $F^{-1}(\{a\})$ is unbounded in $\delta$. We will define a ${<}\kappa$-supported, $\delta$-length iteration of posets $\mathbb P_\alpha$ of size ${<}\delta$. It follows that we can, via coding, think of each $\mathbb P_\alpha$ as an element of $[\delta]^{<\delta}$. Also, if $\dot{a}$ is a $\mathbb P_{\alpha}$ name for a subset of $\lambda$ for some $\lambda < \delta$ then, using a standard nice names argument, we can also think of $\dot{a}$ as coded by an element of $[\delta]^{<\delta}$. Similarly, using standard coding arguments if $H$ is a group of size ${<}\delta$ then it can be coded into an element of $[\delta]^{<\delta}$.

The forcing is now defined as follows. Let $\mathbb P_0$ be the trivial poset. At limit stages we take ${<}\kappa$- sized supports. Suppose $\mathbb P_\alpha$ has been defined as have names $\dot{\rho}_\alpha$ and $\dot{B}_\alpha$ where $1_\alpha \forces \dot{\rho}_\alpha :\dot{B}_\alpha \to S(\check{\kappa})$ induces a $\kappa$-cofinitary representation. 

\noindent \underline{Case 1}: $F(\alpha)$ codes a $\mathbb P_\beta$ nice name $\dot{H}_\alpha$ for an element of $[\delta]^{<\delta}$ which codes a group for some $\beta < \alpha$. In this case let $\dot{\mathbb Q}_\alpha$ be the $\mathbb P_\alpha$ name for the forcing notion $\mathbb Q^\kappa_{\dot{H}_\alpha, \dot{\rho}_\alpha}$. Let $\mathbb P_{\alpha + 1} = \mathbb P_\alpha * \dot{\mathbb Q}_\alpha$. Finally let $\dot{\rho}_{\alpha+1}$ be the name for the generic mapping added by $\dot{\mathbb Q}_\alpha$ which embeds $\dot{H}_\alpha$ into a $\kappa$-cofinitary group extending the image of $\dot{\rho}_\alpha$ and let $\dot{B}_{\alpha+1}$ name an arbitrary set consisting of the disjoint union of $\dot{B}_\alpha$ and a set of the same cardinality as $\dot{H}_\alpha$.

\noindent \underline{Case 2}: Otherwise. In this case let $\dot{Q}_\alpha$ be the $\mathbb P_\alpha$ name for the poset $\mathbb Q^\kappa_{A, \rho}$ where $A$ is an arbitrary set of set $\kappa^+$ (say $\kappa^+$ itself). Let $\mathbb P_{\alpha+1}, \dot{\rho}_{\alpha+1}$ and $\dot{B}_{\alpha+1}$ be defined as in Case 1. 

We claim that $\mathbb P_\delta$ is the required poset. Clearly it is ${<}\kappa$-closed and $\kappa^{+}$-c.c. and adds a $\kappa$-cofinitary group, call it $G$. To see that $G$ embeds every group of cardinality ${<}\delta$, let $\dot{H}$ be a $\mathbb P_\delta$ name for a group of size $\lambda$ for some $\lambda < \delta$. Without loss we can assume that $\dot{H}$ is a nice name for a subset of $\lambda$. It follows that in fact $\dot{H}$ was added by some $\mathbb P_\beta$ for $\beta< \delta$. Moreover since the preimage of $F^{-1}(\{\dot{H}\})$ is unbounded in $\delta$, there is an $\alpha > \beta$ so that $F(\alpha) = \dot{H}$ and so at stage $\alpha$ we forced that $\dot{H}$ embeds into $G$. To see that $G$ is maximal, suppose $\dot{f}:\kappa \to \kappa$ is a $\mathbb P_\delta$ name for a permutation. By standard arguments $\dot{f}$ was added by some $\mathbb P_\alpha$ for $\alpha < \delta$. At a later stage, say $\beta > \alpha$ we were in case 2. Let $G_{\beta+1}$ be the group added by $\mathbb P_{\beta + 1}$. By the properties of the forcing $\mathbb Q^\kappa_{A, \rho}$ we have that $V^{\mathbb P_{\beta+1}} \models$``$G_{\beta+1}$ is a $\kappa$-maximal cofinitary group" and, in particular, it follows that either $\dot{f}$ is forced to be in $G_{\beta+1}$ or else there is a word $w(x) \in W_{G_{\beta + 1}}(x)$ so that $w(\dot{f})$ is forced to not be $\kappa$-cofinitary. In either case it follows that the same holds in $G$ (working in $V[G]$) from which the maximality of $G$ follows.
\end{proof} 

By interweaving the proofs of the above result and Theorem \ref{orbits2} we can also obtain the following.
\begin{corollary}
Assume $\GCH$ and fix $\delta > \kappa$ regular. Given any partition of $\kappa$ into $\lambda$ many pieces $\{O_\xi\; | \; \xi < \lambda\}$ for some $\lambda < \kappa$ there is a cofinality preserving forcing $\mathbb P$ which forces that $2^\kappa = \delta$ and there is a $\kappa$-maximal cofinitary group $G$ which is universal for groups of size ${<} \delta$ and $\mathsf{ORB}(G) = \{O_\xi \; | \; \xi < \lambda\}$.
\end{corollary}

Kastermans' original theorem on the $\omega$ case used $\MA$ as opposed to obtaining the universal group by brute force (pun intended). We would like to obtain the same here using $\mfp(\kappa) = 2^\kappa$ in place of $\MA$. However, since we need to assume that $|H| \leq \kappa$ in order to ensure we can apply the forcing axiom characterization we only obtain the weaker result that $G$ can be universal for groups of size at most $\kappa$. Specifically we have the following.

\begin{theorem}
If $\mathfrak{p}(\kappa) = 2^\kappa$ then there is a $\kappa$-maximal cofinitary group which embeds every group of size $\kappa$. 
\end{theorem}

\begin{proof}
Enumerate $S(\kappa) = \{f_\alpha \;  | \; \alpha < 2^\kappa\}$ and the collection of all groups with domain $\kappa$ as $\{H_\alpha \; | \; \alpha < 2^\kappa\}$. Our group will be constructed transfinitely. We will inductively define an increasing, continuous chain of groups $G_\alpha$ so that for each $\alpha$ $H_\alpha$ embeds into $G_\alpha$ and either $g_\alpha \in G_\alpha$ or $\langle G_\alpha, g_\alpha\rangle$ is not $\kappa$-cofinitary. Our group will be $G : = \bigcup_{\alpha < 2^\kappa} G_\alpha$ as before. 

At stage $0$ first use $\mathbb Q^\kappa_{H_0}$ to build a $\phi_0$ so that $\phi_0:H_0 \to S(\kappa)$ embeds a $\kappa$-cofinitary copy of $H_0$ in $S(\kappa)$. Next, if $g_0$ is hitable with respect to $\phi_0(H_0)$ then use Lemma \ref{addahit} to find a $g_0$ so that $\langle \phi_0(H_0), g_0 \rangle$ is $\kappa$-cofinitary and $f_0$ cannot be added to any subgroup of $S(\kappa)$ containing $g_0$ without killing $\kappa$-cofinitariness. If $f_0$ is not hitable then let $g_0$ be the identity and in either case let $G_0 := \langle \phi_0(H_0), g_0\rangle$.

Now suppose we have constructed a $\kappa$-cofinitary group $G_\alpha$ so that for all $\beta \leq \alpha$ there is an embedding $\phi_\beta: H_\beta \to G_\alpha$, and for every $\beta \leq \alpha$ either $g_\beta \in G_\beta$ or $\langle G_\beta, g_\beta\rangle$ is not $\kappa$-cofinitary. Let $\rho_\alpha:B \to S(\kappa)$ be a mapping which induces a $\kappa$-cofinitary representation equal to $G_\alpha$. Now use $\mathbb Q^\kappa_{H_{\alpha+1}, \rho_\alpha}$ to find a $\phi_{\alpha+1}$ so that $\phi_{\alpha+1}$ embeds $H_{\alpha+1}$ into a $\kappa$-cofinitary extension of $G_\alpha$. Finally if $f_{\alpha+1}$ is hitting with respect to $\langle {\rm Im}(\phi_{\alpha+1}), G_\alpha\rangle$ then use Lemma \ref{addahit} as in stage $0$ to find a $g_{\alpha+1}$. Finally let $G_{\alpha+1} = \langle G_\alpha, {\rm Im}(\phi_{\alpha+1}), g_{\alpha+1}\rangle$. 

This completes the construction. Let $G = \bigcup_{\alpha<2^\kappa} G_\alpha$. Clearly this embeds every group of size $\kappa$ and, by the same argument as was used in Theorem \ref{orbits2} it will be a $\kappa$-maximal cofinitary group.
\end{proof}

Again this theorem can be proved with the added assumption that the universal $G$ has any particular set of ${<}\kappa$ many orbits we like.

\section{Conclusion and Open Questions}

We finish by recording some open questions on the structure of $\kappa$-maximal cofinitary groups. The first concerns getting a converse of Main Theorem \ref{mainthm2}.
\begin{question}
Is it consistent that there is a partition of $\kappa$ of size ${<}\kappa$ which is not the set of orbits of a $\kappa$-maximal cofinitary group?
\end{question}
This seems to be unknown even in the $\omega$ case and would be extremely interesting to investigate further.

The next has to do with the analogue of another theorem of Kastermans. Kastermans also showed in \cite[Theorem 8]{Isomofmcgs} that it's consistent that there is a locally finite maximal cofinitary group. The obvious analogue of this result for $\kappa$, that there is a locally ${<}\kappa$ $\kappa$-maximal cofinitary group, is trivially true since any group of size $>\kappa$ is locally ${<}\kappa$ for any uncountable $\kappa$. Therefore the following question is more appropriate and seems to represent the first place that there many be a divergence in the theories of maximal cofinitary groups on $\omega$ and on $\kappa$.

\begin{question}
Is it consistent that there is a locally finite $\kappa$-maximal cofinitary group?
\end{question}

Finally we repeat the observation noted after the proof of Theorem \ref{orbits1} in a question form.

\begin{question}
Fix $\mu < \kappa$ and define a group to be $(\kappa, \mu)$-cofinitary if it is a subgroup of $S(\kappa)$ all of whose non-identity elements have less than $\mu$ many fixed points. How do the maximal $(\kappa, \mu)$-cofinitary groups differ? In particular, can the associated cardinal characteristics be different? Can there be $\mu_0 < \mu_1 < \kappa$ and $G_0$ a $(\kappa, \mu_0)$-maximal cofinitary group, $G_1$ a $(\kappa, \mu_1)$-maximal cofinitary group so that $G_0 \cong G_1$?
\end{question}


\end{document}